\documentclass[12pt, a4]{amsart}
\usepackage{fullpage}
\usepackage{amsfonts}
\usepackage{amsmath}
\usepackage{amssymb}
\usepackage{braket}
\usepackage{MnSymbol} 

\newcommand{\R}{\mathbb{R}}

\theoremstyle{plain}
\newtheorem{theorem}{Theorem}[section]

\theoremstyle{definition}
\newtheorem{definition}[theorem]{Definition}
\newtheorem{notation}[theorem]{Notation}
\newtheorem{example}[theorem]{Example}

\newtheorem{corollary}[theorem]{Corollary}
\numberwithin{equation}{section}

\begin{document}

\title[Eigenvalues of a third order BVP]{Eigenvalues of a third order BVP subject to functional BCs} 

\author[G. Infante]{Gennaro Infante}
\address{Gennaro Infante, Dipartimento di Matematica e Informatica, Universit\`{a} della
Calabria, 87036 Arcavacata di Rende, Cosenza, Italy}%
\email{gennaro.infante@unical.it}%

\author[P. Lucisano]{Paolo Lucisano}
\address{Paolo Lucisano,
Liceo Scientifico Statale ``Alessandro Volta'', Via Modena S. Sperato, 89133 Reggio Calabria, Italy}%
\email{paolo.lucisano@hotmail.com}%

\begin{abstract} 
We discuss the existence of eigenvalues for a third order boundary value problem subject to functional boundary conditions and higher order derivative dependence in the nonlinearities. 
We prove the existence of positive and negative eigenvalues and provide a localization of the corresponding eigenfunctions in terms of their norm.
The methodology involves a version of the classical Birkhoff--Kellogg theorem. We illustrate the applicability of the theoretical results in an example.
\end{abstract}



\maketitle

\section{Introduction}
The solvability of third order boundary value problems~(BVPs), under a variety of nonlocal boundary conditions, has been investigated in the past few years by a number of authors; we refer the reader to the 
papers~\cite{and-dav, ca-lo-min, acgi2, jg-jw-na, jg-by1, jg-by2, gi-pp-MB, Smi1, Smi2, Smi3, Gab-Mira1, Gab-Mira2, Webb-Con} 
and the references therein. In particular, in the  manuscripts~\cite{Smi1,Smi2}
Smirnov studied the solvability of the~BVP
\begin{equation}
\label{BVP-Smirnov}
\begin{cases}
u'''(t)+f(t,u(t),u'(t),u''(t))=0, \  t\in (0,1), \\
u(0)=u(1)=
\int_{0}^{1} u(t)\,dt=0.
\end{cases}
\end{equation}
The methodology in~\cite{Smi1} relies on metric fixed point theory, namely the classical Banach fixed point theorem and a recent fixed point theorem due to Rus, while the approach in~\cite{Smi2} is of topological nature, and is based on the celebrated Leray--Schauder continuation principle. 

Here we discuss a parameter-dependent variation of the BVP~\eqref{BVP-Smirnov}, under the presence of higher order derivatives and additional functional conditions at the boundary, namely 
\begin{equation}
\label{BVPintro}
\begin{cases}
u'''(t)+\lambda f(t,u(t),u'(t),u''(t))=0 , \ t \in (0,1), \\
u(0)=\lambda H_1[u],\ u(1)=\lambda H_2[u],\
\int_{0}^{1} u(t)\,dt=0,
\end{cases}
\end{equation}
where $f$ is a continuous function and $H_1, H_2$ are suitable continuous functionals and $\lambda$ is a parameter. Our approach relies on a version of the classical  Birkhoff-Kellogg invariant direction theorem, combined with ideas from the recent paper~\cite{giaml22}, where a similar parameter dependent problem was investigated  in the case of a cantilever equation. Note that, due to presence of the nonlocal condition  
$\int_{0}^{1} u(t)\,dt=0$, positive solutions of the BVP~\eqref{BVPintro} do not exist; therefore, instead of seeking solutions within a cone of positive functions as in~\cite{giaml22}, we utilize balls centered at the origin in the space $C^2[0,1]$. We prove existence of eigenvalues of opposite sign for the BVP~\eqref{BVPintro}, with a  localization of the associated eigenfunctions in terms of their  $C^2$-norm.
We illustrate the applicability of our theoretical results in an example, where we compute the constants that occur in our theory.  Our results are new and complement the ones in~\cite{Smi1,Smi2}.

\section{Eigenvalues existence}
We begin this Section by associating to the BVP~\eqref{BVPintro} a suitable perturbed Hammerstein integral equation. 

Firstly we recall that it is known (see for example Proposition~1 of~\cite{Smi1}) that, for $h\in C[0,1]$, the unique solution of the linear BVP
$$
\begin{cases}
u^{'''}(t)+h(t)=0,\ t \in (0,1),\\ u(0)=u(1)=\int_{0}^{1}u(t)\,dt=0,
\end{cases}
$$
is given by
\begin{equation*}
u(t)=\int_{0}^{1} k(t,s) h(s)\,ds,
\end{equation*}
where
\begin{equation*}
k(t,s)=\dfrac{1}{2}
\begin{cases}
s^2(1-t)(3t-2ts-1), & \text{if} \ 0\le s\le t\le 1, \\
t(1-s)^2(2ts+t-2s), & \text{if} \ 0 \le t\le s\le 1.
\end{cases}
\end{equation*} 
By direct computation we obtain
\begin{equation*}
\dfrac{\partial}{\partial t} k(t,s)=
\begin{cases}
s^2[s(2t-1)-3t+2], & \text{if} \ 0\le s\le t\le 1, \\
(1-s)^2[s(2t-1)+t], & \text{if} \ 0 \le t\le s\le 1,
\end{cases}  
\end{equation*}
and
\begin{equation*}
\dfrac{\partial^2}{\partial t^2} k(t,s)=
\begin{cases}
s^2(2s-3), & \text{if} \ 0\le s< t\le 1, \\
(1-s)^2(2s+1), & \text{if} \ 0 \le t< s\le 1.
\end{cases}
\end{equation*}
We observe that
\begin{equation*}
\frac{\partial^2 k}{\partial t^2}(0,s) \geq 0\ \text{for}\ s\in (0,1] ,
\end{equation*}
and
\begin{equation*}
\frac{\partial^2 k}{\partial t^2}(1,s) \leq 0\ \text{for}\ s \in \ [0,1).
\end{equation*}

Secondly, note that 
\begin{equation*}
\gamma_1 (t)=1-4t+3t^2
\end{equation*} is the unique solution of the BVP
\begin{equation*}
\gamma_1^{'''}(t)=0,\; \gamma_1(0)=1,\; \gamma_1(1)=0,\; \int_{0}^{1} \gamma_1(t)\,dt=0,
\end{equation*}
while
\begin{equation*}
\gamma_2 (t)=-2t+3t^2
\end{equation*} is the unique solution of the BVP
\begin{equation*}
\gamma_2^{'''}(t)=0,\; \gamma_2(0)=0,\; \gamma_2(1)=1,\;  \int_{0}^{1} \gamma_2(t)\,dt=0.
\end{equation*}
A direct calculation  yields
\begin{equation*}
\gamma_1^{''}(t)=\gamma_2^{''}(t)=6, \  \text{for}\ t \in[0,1].
\end{equation*}

With the above ingredients at our disposal, we associate to the BVP~\eqref{BVPintro} the perturbed Hammerstein integral equation \begin{equation}
\label{PHIE}
u(t)=\lambda\Bigl(\gamma_1(t)H_1[u]+\gamma_2(t)H_2[u]+\int_{0}^{1}k(t,s)f(s,u(s),u'(s),u''(s))\,ds\Bigr).
\end{equation} 
We work in the space $C^2[0,1]$ endowed with the norm
$$\|u\|_2:=\max_{j=0,1,2}\{\|u^{(j)}\|_\infty\},\ \text{where}\ \|w\|_\infty=\sup_{t\in [0,1]}|w(t)|.$$

\begin{definition} We say that $\lambda$ is an \emph{eigenvalue} of the BVP \eqref{BVPintro} with a corresponding \emph{eigenfunction} $u\in C^2[0,1]$ if $\|u\|_2>0 $ and the pair $(u,\lambda)$ satisfies the perturbed Hammerstein integral equation~\eqref{PHIE}.
\end{definition}
In what follows, we use of the following notation.
\begin{notation}
Let $(X,\vert\vert \cdot \vert\vert)$ be a real Banach space and let $\hat{F}:X\to X$ be continuous. Fix $\rho\in (0, +\infty)$ and set
$$B_\rho(X)=\Set{x\in X \,\big\mid\,\, \vert\vert x \vert\vert < \rho},\ \overline{B_\rho(X)}=\Set{x\in X \,\big\mid\,\, \vert\vert x \vert\vert \le \rho},\ \partial B_\rho(X)=\Set{x\in X \,\big\mid \,\, \vert\vert x \vert\vert = \rho},$$
%\item $\Lambda(F)=\Set{\lambda\in\K \,\big\mid \, \exists x\in X \, s.t. \, F(x)=\lambda x};$
 $$\Lambda_\rho(\hat{F})=\Set{\lambda\in\mathbb{R}\,\big\mid \, \text{there exists} \ x\in \partial B_\rho(X) \, \text{such that} \, \hat{F}(x)=\lambda x}.$$
\end{notation}
We can now recall the following version of the Birkhoff-Kellogg theorem.
\begin{theorem}[\cite{applbook}, Theorem 10.2]
\label{BKthm}
Let $X$ be an infinite dimensional real Banach space, and let $\hat{F}\colon \overline{B_\rho(X)}\to X$ be a compact operator such that:
\begin{equation*}
\inf_{\vert\vert x \vert\vert = \rho} \vert\vert \hat{F}(x) \vert\vert > 0.
\end{equation*}
Then there exist
$
\lambda_+,\lambda_-\in{\Lambda_\rho(\hat{F})}$, with $\lambda_+>0,\lambda_-<0.$
\end{theorem}
The following Theorem provides an existence result for two eigenvalues of opposite sign, with corresponding eigenvalues 
possessing a fixed norm.

\begin{theorem}
\label{mainthm}
Let $\rho \in (0,+\infty)$ and set $\Pi_\rho=[0,1]\times[-\rho,\rho]^3$. Assume that the following conditions hold.
\begin{itemize}
\item[$(1)$] $f\in C^0(\Pi_\rho,\R)$ and there exists $\underline{\delta}_\rho\in C^0([0,1],\R_+)$ such that one of the following conditions holds.
\begin{itemize}
\item[$(1a)$] 
$
f(t,u,v,w)\ge \underline{\delta}_\rho(t),\ \text{for every}\ (t,u,v,w)\in \Pi_{\rho}.$
\item[$(1b)$] 
$-f(t,u,v,w)\ge \underline{\delta}_\rho(t),\ \text{for every}\ (t,u,v,w)\in \Pi_{\rho}.$
\end{itemize}
\item[$(2)$] $H_1,H_2: \overline{B_\rho(C^2[0,1])}\to \mathbb{R}$ are continuous and bounded. Let $\underline{\eta_1}_\rho, \underline{\eta_2}_\rho \in \mathbb{R}$ be such that $$H_i[u]\ge \underline{\eta_i}_\rho,\ \text{for every}\ u \in  \partial B_\rho(C^2[0,1])\ \text{and}\ i\in \{1,2\}.$$
\item[(3)] One of the following two condition holds.
\begin{itemize}
\item[$(3a)$] 
Under the assumption $(1a)$ the inequality 
\begin{equation}\label{BKineqa}
6\big(\underline{\eta_1}_\rho+\underline{\eta_2}_\rho\big)+\int_{0}^{1} (1-s)^2(2s+1)\underline{\delta}_\rho(s)\,ds>0
\end{equation}
is satified.
\item[$(3b)$]
Under the assumption $(1b)$ the inequality 
\begin{equation}\label{BKineqb}
6\big(\underline{\eta_1}_\rho+\underline{\eta_2}_\rho\big)+\int_{0}^{1} s^2(3-2s)\underline{\delta}_\rho(s)\,ds>0
\end{equation}
is satified.
\end{itemize} 
\end{itemize} 
Then the BVP~\eqref{BVPintro} has a positive eigenvalue $\lambda_\rho^+$ with an associated eigenfunction $u_\rho^+\in \partial B_\rho(C^2[0,1])$ and a negative eigenvalue $\lambda_\rho^-$ with an associated eigenfunction 
$u_\rho^-\in \partial B_\rho(C^2[0,1]).$
\end{theorem}

\begin{proof}
Let $T\colon \overline{B_\rho(C^2[0,1])}\to C^2[0,1]$ be defined by 
\begin{equation*}
Tu(t)=Fu(t)+\Gamma u(t),
\end{equation*}
where 
\begin{equation*}
Fu(t)=\int_{0}^{1} k(t,s)f(s,u(s),u'(s),u''(s))\,ds,
\end{equation*}
and
\begin{equation*}
\Gamma u(t)=\gamma_1(t) H_1[u]+\gamma_2(t) H_2[u].
\end{equation*}
Note that $F$ is compact as consequence of the Arzel\`{a}-Ascoli theorem (see~\cite{Webb-Cpt}),  
while the compactness of $\Gamma$ is due to the fact that it is a continuous finite rank operator; therefore $T$ is compact. We now distinguish two cases.

Firstly assume that $(3a)$ holds. Then, for  $u\in\partial B_\rho(C^2[0,1])$, we have that
\begin{multline}\label{pfa}
 \vert\vert Tu \vert\vert_2\ge \vert\vert (Tu)'' \vert\vert_\infty\ge \big\mid (Tu)''(0)\big\mid\ge(Tu)''(0)  \\
 =6\big(H_1[u]+H_2[u]\big)+\int_{0}^{1} \bigg[\frac{\partial^2 }{\partial t^2} k(t,s)\bigg]_{t=0}f(s,u(s),u'(s),u''(s))\,ds \\
=6\big(H_1[u]+H_2[u]\big)+\int_{0}^{1} (1-s)^2(2s+1)f(s,u(s),u'(s),u''(s))\,ds \\
\ge 6\big(\underline{\eta_1}_\rho+\underline{\eta_2}_\rho\big)+\int_{0}^{1} (1-s)^2(2s+1)\underline{\delta}_{\rho}(s) \,ds.
\end{multline}

Now suppose that  $(3b)$ holds. Then, for  $u\in\partial B_\rho(C^2[0,1])$ we have that
\begin{multline}\label{pfb}
 \vert\vert Tu \vert\vert_2\ge \vert\vert (Tu)'' \vert\vert_\infty\ge \big\mid (Tu)''(1)\big\mid\ge(Tu)''(1)\\
 =6\big(H_1[u]+H_2[u]\big)+\int_{0}^{1} \bigg[\frac{\partial^2 }{\partial t^2} k(t,s)\bigg]_{t=1}f(s,u(s),u'(s),u''(s))\,ds \\
= 6\big(H_1[u]+H_2[u]\big)+\int_{0}^{1} s^2(2s-3)f(s,u(s),u'(s),u''(s))\,ds \phantom{spac}\\
=6\big(H_1[u]+H_2[u]\big)+\int_{0}^{1} s^2(3-2s)\big(-f(s,u(s),u'(s),u''(s))\big)\,ds \\
\ge 6\big(\underline{\eta_1}_\rho+\underline{\eta_2}_\rho\big)+\int_{0}^{1} s^2(3-2s)\underline{\delta}_{\rho}(s)\,ds. 
\end{multline}  

Note that the RHS of the inequalities \eqref{pfa} and \eqref{pfb} 
do not depend on the particular $u$ chosen. In both cases, we have that 
\begin{equation*}
\inf\limits_{\vert\vert u \vert\vert_2=\rho} \vert\vert Tu \vert\vert_2>0,
\end{equation*}
and the result follows by Theorem~\ref{BKthm}.
\end{proof}
The following Corollary provides an existence result for uncountably many couples of eigenvalues--eigenfunctions.
\begin{corollary}
\label{CorEig}
If in addition to the hypotheses of Theorem~\ref{mainthm} we assume that $\rho$ can be chosen arbitrarily in $(0,+\infty)$, then for every $\rho$ there exist two eigenfunctions $u_{\rho}^{+}\in\partial B_\rho(C^2[0,1])$ and $u_{\rho}^{-}\in\partial B_\rho(C^2[0,1])$ of the BVP $(\ref{BVPintro})$ to which corresponds, respectively, an eigenvalue $\lambda_{\rho}^{+}>0$ and an eigenvalue $\lambda_{\rho}^{-}<0$.
\end{corollary}
We conclude with an example that illustrates the applicability of the previous theoretical results.
\begin{example}
We consider the BVP
\begin{equation}
\label{BVPex}
\begin{cases}
u'''(t)+\lambda te^{\mid u(t)\mid}[1+(u''(t))^2]=0, \ \text{for $t\in(0,1)$}, \\
u(0)=\dfrac{\lambda}{1+u^2\big(\frac{1}{2}\big)}, \\
u(1)=\dfrac{\lambda}{40}\sin(\int_{0}^{1}t^3u''(t)dt), \\
\int_{0}^{1} u(t)dt=0.
\end{cases}
\end{equation}
Fix $\rho\in (0,+\infty)$.
Note that we have
\begin{align*}
f(t,u,v,w)\ge t:&= \underline{\delta}_\rho(t),\ \text{for every}\ (t,u,v,w)\in \Pi_{\rho},\\
1\geq H_1[u] \geq \dfrac{1}{1+\rho^2}:&= \underline{\eta_1}_\rho,\ \text{for every}\ u\in  \partial B_\rho(C^2[0,1]),\\
\dfrac{1}{40}\geq H_2[u] \geq -\dfrac{1}{40}:&=\underline{\eta_2}_\rho,\ \text{for every}\ u\in  \partial B_\rho(C^2[0,1]).
\end{align*}
Thus the inequality \eqref{BKineqa} reads
$$
6\Big( \dfrac{1}{1+\rho^2}-\dfrac{1}{40}\Big) +\int_{0}^{1} (1-s)^2(2s+1)s\,ds=\dfrac{6}{1+\rho^2}>0.
$$
Therefore we can apply Corollary~\ref{CorEig}, obtaining uncountably many pairs of positive and negative eigenvalues for the BVP~\eqref{BVPex}, with corresponding eigenfunctions that posses a localized $C^2$-norm.
\end{example}
\section*{Acknowledgements}
The authors would like to thank the anonymous Referee for the careful reading of the manuscript and
the constructive comments.
G.~Infante is a member  of the ``Gruppo Nazionale per l'Analisi Matematica, la Probabilit\`a e le loro Applicazioni'' (GNAMPA) of the Istituto Nazionale di Alta Matematica (INdAM) and of the UMI Group TAA ``Approximation Theory and Applications''.  
G.~Infante was partly funded by the Research project of MUR - Prin~2022 “Nonlinear differential problems with applications to real phenomena” (Grant Number: 2022ZXZTN2).


\begin{thebibliography}{999}

\bibitem{and-dav}
D. R. Anderson and J. M. Davis, Multiple solutions and eigenvalues for third-order right focal boundary value problems, 
\textit{J. Math. Anal. Appl.}, \textbf{267} (2002), 135--157.

\bibitem{applbook} J. Appell, E. De Pascale and A. Vignoli, \textit{Nonlinear spectral
theory}, Walter de Gruyter \& Co., Berlin, (2004).

\bibitem{ca-lo-min}
A. Cabada, L. L{\'o}pez-Somoza and F. Minh{\'o}s, Existence, non-existence and multiplicity results for a third order eigenvalue three-point boundary value problem, 
\textit{J. Nonlinear Sci. Appl.}, \textbf{10} (2017), 5445--5463.

\bibitem{acgi2} A. Calamai and G. Infante, 
An affine Birkhoff--Kellogg type result in cones with applications to functional differential equations,
\textit{Math.\ Meth.\ Appl.\ Sci.}, \textbf{46} (2023), 11897--11905.

\bibitem{jg-jw-na} 
J. R. Graef and J. R. L. Webb, Third order boundary value problems with nonlocal boundary conditions,
\textit{Nonlinear Anal.}, \textbf{71} (2009), 1542--1551.

\bibitem{jg-by1} J. Graef and B. Yang, Multiple positive solutions to a three point third order boundary value
problem, \textit{Discrete Contin. Dyn. Syst.}, \textbf{vol. suppl.} (2005), 337--344.

\bibitem{jg-by2}
J. Graef and B. Yang, Positive solutions of a third order nonlocal boundary value problem,
\textit{Discrete Contin. Dyn. Syst. Ser. S.}, \textbf{1} (2008), 89--97.

\bibitem{giaml22}
G. Infante, On the solvability of a parameter-dependent cantilever-type BVP, \textit{Appl. Math. Lett.}, \textbf{132} (2022), 108090.

\bibitem{gi-pp-MB}
 G. Infante and P. Pietramala, A third order boundary value problem subject to nonlinear boundary conditions,
\textit{Math. Bohem.}, \textbf{135} (2010), 113--121. 

\bibitem{Smi1} S. Smirnov, Existence of a unique solution for a third-order boundary value problem with nonlocal conditions of integral type, 
\textit{Nonlinear Anal., Model. Control},  \textbf{26} (2021), 914--927.
		
\bibitem{Smi2} 	S. Smirnov, 
Existence of sign-changing solutions for a third--order boundary value problem with nonlocal conditions of integral type,
\textit{Topol. Methods Nonlinear Anal.}, \textbf{62} (2023), 377--384.

\bibitem{Smi3} 	S. Smirnov, 
Multiplicity of positive solutions for a third-order boundary value problem with nonlocal conditions of integral type,
\textit{Miskolc Math. Notes}, \textbf{25} (2024), 967--975.

\bibitem{Gab-Mira1}
G. Szajnowska and M. Zima,
Positive solutions to a third order nonlocal boundary value problem with a parameter,
\textit{Opusc. Math.}, \textbf{44} (2024), 267--283.
 
 \bibitem{Gab-Mira2}
G. Szajnowska and M. Zima,
A fixed point index approach to a third order nonlocal boundary value problem 
\textit{Topol. Methods Nonlinear Anal.}, to appear.

 \bibitem{Webb-Con}
J. R. L. Webb,
Higher order non-local $(n-1,1)$ conjugate type boundary value problems, \textit{Mathematical models in engineering, biology and medicine}, \textbf{vol. 1124} (2009),  332--341.


 \bibitem{Webb-Cpt}
J. R. L. Webb, Compactness of nonlinear integral operators with discontinuous and with singular kernels, \textit{J. Math. Anal. Appl.}, \textbf{509} (2022),  Paper No. 126000, 17 pp.



\end{thebibliography}
\end{document}